\documentclass[reqno]{amsart}
\usepackage{amsfonts,graphicx,float}
\usepackage[arrow,matrix]{xy}

\makeatletter                                                  %
\def\th@plain{\slshape}                                        %
\makeatother                                                   %

\newcommand{\ooi}{[0,1)}
\newcommand{\ooii}{[0,1]}

\newcommand{\Zbb}{\mathbb{Z}}
\newcommand{\Qbb}{\mathbb{Q}}
\newcommand{\Rbb}{\mathbb{R}}
\newcommand{\Cbb}{\mathbb{C}}

\newcommand{\one}{{\rm 1\mskip-4mu l}}
\newcommand{\ud}{\,\mathrm{d}}
\newcommand{\bs}{\backslash}

\newcommand{\p}{_{\ge0}}
\newcommand{\pp}{_{>0}}
\newcommand{\m}{^{-1}}
\newcommand{\md}{^{-2}}

\newcommand{\Fcal}{\mathcal{F}}
\newcommand{\Hcal}{\mathcal{H}}
\newcommand{\Scal}{\mathcal{S}}

\newcommand{\abs}[1]{\lvert#1\rvert}
\newcommand{\newword}[1]{\textsl{#1}}

\newcommand{\bb}[1]{\mathbf{#1}}
\newcommand{\set}[1]{\{ #1 \}}
\newcommand{\cvect}[2]{\bigl(\begin{smallmatrix}#1\\#2\end{smallmatrix}\bigr)}
\newcommand{\rvect}[2]{(#1\;\;#2)}
\newcommand{\smatrix}[4]{\bigl(\begin{smallmatrix}#1&#2\\#3&#4\end{smallmatrix}\bigr)}

\DeclareMathSymbol{\upharpoonright}{\mathrel}{AMSa}{"16}
\let\restriction\upharpoonright

\DeclareMathOperator{\den}{den}

\DeclareMathOperator{\PSL}{PSL}
\DeclareMathOperator{\nngg}{ng}
\DeclareMathOperator{\Cone}{Cone}
\DeclareMathOperator{\id}{id}
\DeclareMathOperator{\re}{re}
\DeclareMathOperator{\im}{im}
\DeclareMathOperator{\area}{area}

\theoremstyle{plain}
\newtheorem{theorem}{Theorem}[section]
\newtheorem{lemma}[theorem]{Lemma}

\theoremstyle{definition}
\newtheorem{definition}[theorem]{Definition}

\newtheorem{example}[theorem]{Example}

\begin{document}

\bibliographystyle{plain}

\sloppy

\title[The weighted Farey sequence]{The weighted Farey sequence\\
and a sliding section for the horocycle flow}

\author[]{Giovanni Panti}
\address{Department of Mathematics\\
University of Udine\\
via delle Scienze 206\\
33100 Udine, Italy}
\email{giovanni.panti@uniud.it}

\begin{abstract}
The Farey sequence is the sequence of all rational numbers in the real unit interval, stratified by increasing denominators. A classical result by Hall says that its normalized gap distribution is the same as the distribution of the random variable $\bigl(2\,\zeta(2)\,xy\bigr)\m$ on a certain unit triangle. In this paper we weight the  denominators by an arbitrary piecewise-smooth continuous function, and we characterize the resulting gap distribution as that of a multiple of the above variable, defined on a certain unit pentagon. Our characterization refines previous results by  
Boca, Cobeli and Zaharescu, but employs completely different techniques. Building upon recent work by Athreya and Cheung,
we construct a varying-with-time Poincar\'e section for the horocycle flow on the space of unimodular lattices, and we interpret the weighted Farey sequence as the list of return times to the section. Under an appropriate parametrization, our pentagon appears as the orbit of Hall's triangle under the motion of the section, and basic equidistribution results for long closed horocycles yield explicit formulas for the limit transverse measure.
\end{abstract}

\thanks{\emph{2010 Math.~Subj.~Class.}: 37D40; 11B57.}

\maketitle

\section{Introduction}\label{ref3}

For every $Q=1,2,3,\ldots$, let
$\Fcal(Q)=\set{0=s_0<s_1<s_2<\cdots<s_{n(Q)-1}}$ be a finite subset of the half-open real unit interval $\ooi$. Assume that $\Fcal(1)\subset\Fcal(2)\subset\cdots$, with union dense in $\ooi$. Setting $s_{n(Q)}=1$, the \newword{normalized gap} at $s_i\in\Fcal(Q)$ is
\[
\nngg_Q(s_i)=n(Q)(s_{i+1}-s_i).
\]
If, for every $z\in\Rbb\p$, the limit
\begin{equation}\label{eq1}
H(z)=\lim_{Q\to\infty}\frac{\sharp\set{0\le i<n(Q):\nngg_Q(s_i)\le z}}{n(Q)}
\end{equation}
exists, then we say that the sequence of the $\Fcal(Q)$'s has \newword{cumulative gap distribution}~$H$.

Two extreme cases of this setting are the Heaviside distribution (the distribution of a random variable which is $1$ almost surely), which is easily realizable via ``evenly spaced'' $\Fcal(Q)$'s, and the exponential distribution $1-\exp(-z)$, which is almost surely induced whenever the points of $\Fcal(Q)$ are given by i.i.d.~random variables uniformly distributed on~$\ooi$.

Throughout this paper rational numbers $s=p/q$ are always written in reduced form (i.e., $q>0$ and $p,q$ relatively prime). The \newword{Farey sequence of order $Q$} is the set $\Fcal_\one(Q)$ of all rational numbers in $\ooi$ whose denominator is $\le Q$ (we'll explain the subscript $\one$ shortly). All intervals $[p_i/q_i,p_{i+1}/q_{i+1}]$ between consecutive points of $\Fcal_\one(Q)\cup\set{1}$ are \newword{unimodular} (i.e., $\det\smatrix{p_{i+1}}{p_i}{q_{i+1}}{q_i}=1$), and hence have length $(q_{i+1}q_i)\m$, which is bounded from below by $Q^{-2}$. It is a classical fact that the number of intervals in $\Fcal_\one(Q)$ is asymptotic, for $Q\to\infty$, to $\bigl(2\,\zeta(2)\bigr)^{-1}Q^2$.
This immediately implies that the normalized gaps are bounded from below by $\bigl(2\,\zeta(2)\bigr)\m=3/\pi^2=0.30396\ldots$, so that $H(z)$ exists and has value $0$ for $z\le 3/\pi^2$. This remark is just a fraction of Hall's classical result~\cite{hall70}, according to which the limit~\eqref{eq1} exists for the Farey sequence and agrees with the cumulative distribution $H_\one(z)$ of the random variable
\begin{equation*}
Z_1(x,y)=\frac{1}{2\,\zeta(2)\,xy}.
\end{equation*}
The latter is defined on the space $(\Omega_1,P_1)$, where
$\Omega_1$ is the triangle $\set{(x,y)\in\Rbb\pp^2:x,y\le 1<x+y}$ and $P_1$ is the Lebesgue measure, normalized by $P_1(\Omega_1)=1$. 

Explicit computation gives
\begin{equation*}
H_\one(z)=\begin{cases}
0, & \text{if $z\le 3/\pi^2$;}\\
2-6\bigl(1+\log(\pi^2z/3)\bigr)/(\pi^2z),
& \text{if $3/\pi^2<z\le 12/\pi^2$;} \\
2-6/(\pi^2z)-\sqrt{1-12/(\pi^2z)}+ \\
\quad 12\log(1/2+\sqrt{1/4-3/(\pi^2z)})/(\pi^2z),
& \text{if $12/\pi^2 < z$.}
\end{cases}
\end{equation*}
Differentiating, we obtain the density distribution function
\begin{equation*}
h_\one(z)=\begin{cases}
0, & \text{if $z\le 3/\pi^2$;}\\
6\log(\pi^2z/3)/(\pi^2z^2),
& \text{if $3/\pi^2<z\le 12/\pi^2$;} \\
-12\log\bigl(1/2+\sqrt{1/4-3/(\pi^2z)}\bigr)/(\pi^2z^2),
& \text{if $12/\pi^2 < z$.}
\end{cases}
\end{equation*}

We plot $h_\one(z)$ in Figure~\ref{ref1}; the two points of nondifferentiability correspond to the hyperbola $\set{Z_1(x,y)=z}$ hitting~$\Omega_1$ in the upper right corner (at $z=3/\pi^2$) and in the midpoint of the hypothenuse (at $z=12/\pi^2$).
\begin{figure}[h!]
\includegraphics[width=8cm,height=5cm]{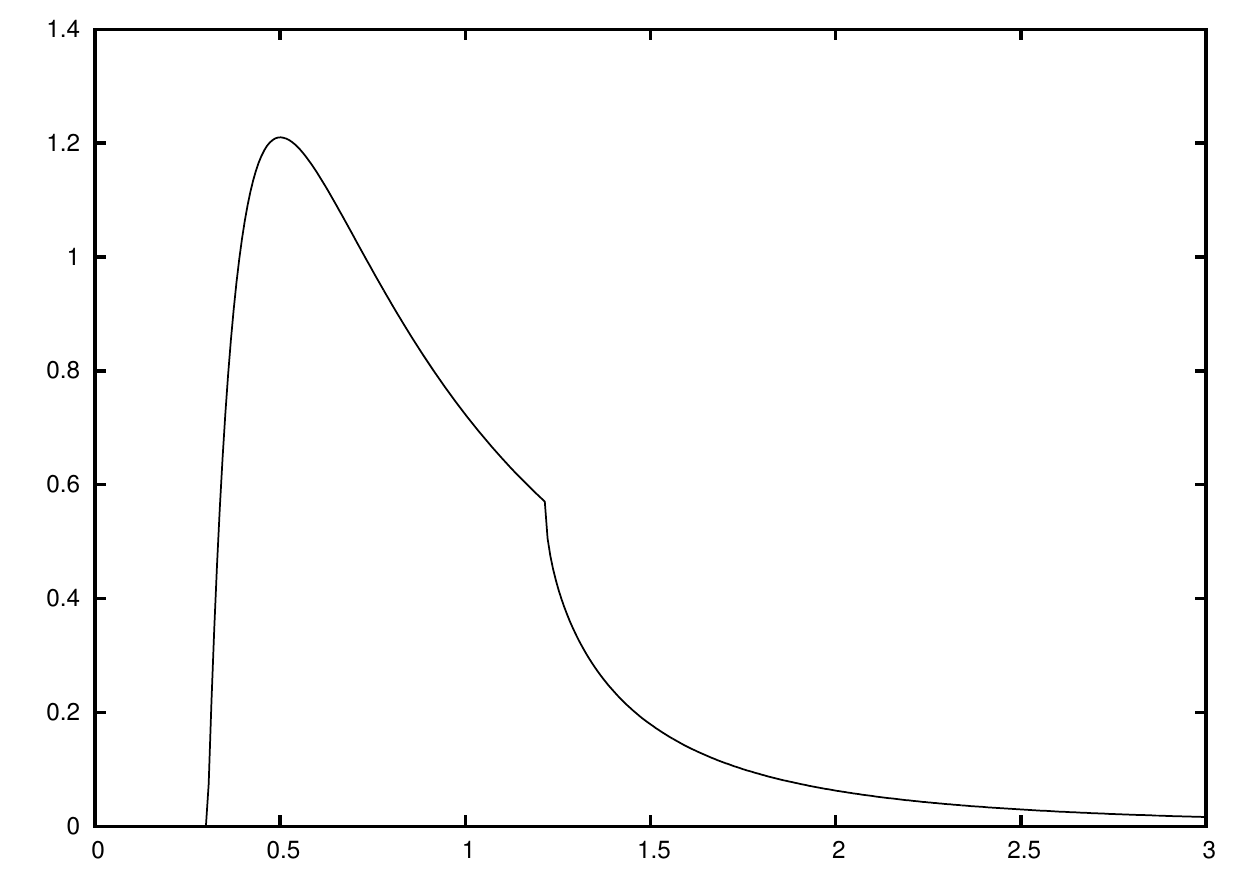}
\caption{}
\label{ref1}
\end{figure}

In this paper we are interested in the statistics of gap distribution when the denominators of rational points are weighted by a fixed function $u:\ooii\to\Rbb\pp$ such that:
\begin{itemize}
\item[(i)] $u$ is continuous;
\item[(ii)] $\ooii$ can be partitioned into finitely many subintervals $[s,s']$ (with $s,s'$ not necessarily rational), overlapping only at endpoints, and such that $u$ is $C^1$ with bounded first derivative on each $(s,s')$.
\end{itemize}
We call such a function a \newword{unit} (the name stems from algebraic logic~\cite{pantiravotti}). The \newword{$u$-denominator} of $s=p/q\in\Qbb\cap \ooii$ is then $\den_u(s)=u(s)\,q$; the constant function~$\one$ is a unit, and $\den_\one(s)$ is the ordinary denominator of $s$. We can then form the \newword{$u$-weighted Farey sequence} $\Fcal_u(Q)$, whose elements are all the rational points in $\ooi$ of $u$-denominator $\le Q$.

We can look at things projectively, by defining the \newword{cone} over $\ooi$ by $\Cone\ooi=\Rbb\pp\cdot\bigl\{\cvect{s}{1}:s\in \ooi\bigr\}\subset\Rbb^2$. Every function $f$ on $\ooi$ gives rise to its \newword{homogeneous correspondent} $\bb{f}:\Cone\ooi\to\Rbb$ by $\bb{f}(\bb{s})=
y\, f(x/y)$, where $\bb{s}=\cvect{x}{y}$.
By defining the homogeneous correspondent of the rational point $s=p/q$ to be $\bb{s}=\cvect{p}{q}$ (we are trying to use consistently boldface type for projective objects, and lightface for affine ones), we immediately see that $\den_u(s)=\bb{u}(\bb{s})$, and the points in $\Fcal_u(Q)$ correspond bijectively to the primitive integer points in $\Cone\ooi\cap\set{\bb{u}\le Q}$.

Summing up, we are looking at the limiting gap distribution ---call it $H_u$--- of the projections of the above primitive integer points on $\ooi\times\set{1}$.
An explicit expression for $H_u$ as a mean over the Hall distribution $H_\one$ is obtained by Boca, Cobeli and Zaharescu in~\cite[Theorem~0.2]{boca_et_al00}; up to changes in parametrization and notation, it is formula~\eqref{eq2} below. Their proof uses incomplete Kloosterman sums and analytic number theory, spanning several pages of delicate computation. Building upon work of Athreya and Cheung~\cite{athreyache14}, we exploit here the properties of the horocycle flow on the space of rank-2 unimodular lattices to provide a short and reasonably self-contained proof of~\eqref{eq2}. Our key technical tool, and the main novelty of this paper, is the use of a ``sliding Poincar\'e section'' for the flow (see \S\ref{ref21}).

We remark that the applicability of dynamical equidistribution to the statistics of primitive lattice points inside an arbitrary star-shaped domain ---even in a higher dimensional setting--- was already pointed out by Marklof in~\cite{marklof13} (paragraph starting at the bottom of p.~50). Although in dimension greater than $1$ the resulting limit distributions are not as visualizable as those in the classical case, these techniques grant the transfer of much information; see~\cite{marklofstr10}, \cite{marklofstr11}, \cite{strombergsson11}. The crux of the matter lies ultimately in a very general result on the equidistribution of Farey sequences on large closed horospheres~\cite[Theorem~6]{marklof10}, in which the test function has an explicit dependence on the sequence points.

Let us state our results; throughout this paper $u$ is a fixed unit. We set
\begin{align*}
v(s)&=u(s)\m,\\
C&=\int_0^1 v(s)^2\ud s \in\Rbb\pp,\\
m(x)&=C\m v(s)^2.
\end{align*}
The explicit expression for $H_u$, to be proved in Theorem~\ref{ref20}, is then
\begin{equation}\label{eq2}
H_u(z)=\int_0^1 H_\one\bigl(m(s)\, z\bigr)\, m(s)\ud s.
\end{equation}
Differentiating under the integral sign is safe here, and we obtain
\begin{equation}\label{eq3}
h_u(z)=\int_0^1 h_\one\bigl(m(s)\, z\bigr)\, m(s)^2\ud s.
\end{equation}

\begin{example}
Let\label{ref5} $u$ be the unit
\[
u(x)=\begin{cases}
(5x+1)/2, &\text{if $0\le x\le 1/5$};\\
(x-2/5)^2+24/25, &\text{otherwise}.
\end{cases}
\]
In Figure~\ref{ref2} we plot the graph of $u$, as well as the histogram of the gap distribution of $\Fcal_u(400)$ against the expression for $h_u(z)$ in~\eqref{eq3}.
\begin{figure}[h!]
\includegraphics[height=4.5cm]{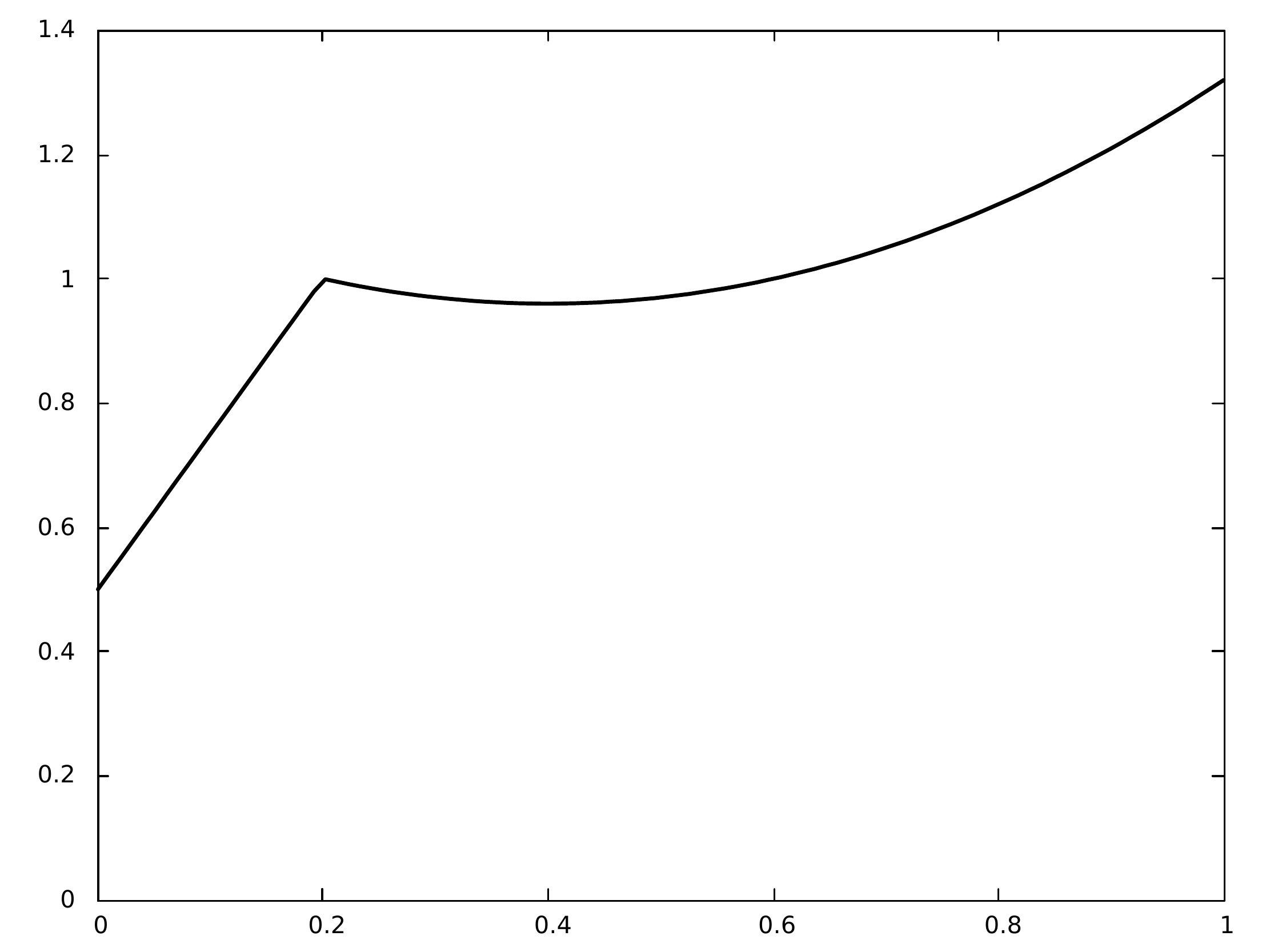}
\includegraphics[height=4.5cm]{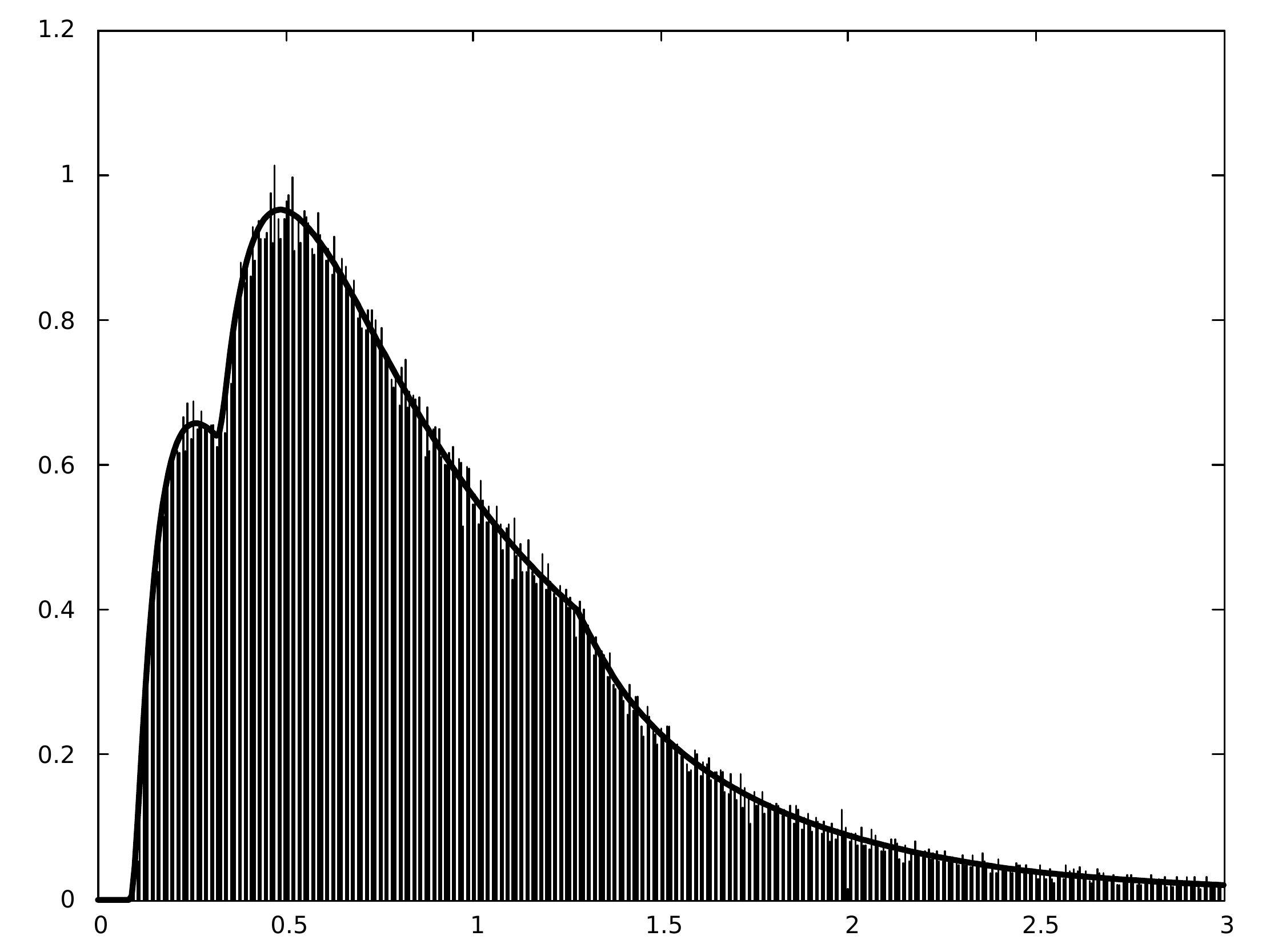}
\caption{}
\label{ref2}
\end{figure}
\end{example}

In Theorem~\ref{ref20} we will express $H_u$ as the cumulative distribution function of the random variable $Z(x,y)=C/(2\,\zeta(2)\,xy)$, with the Hall triangle $\Omega_1$ replaced by a pentagon endowed with an appropriate probability measure; see Figure~\ref{ref29} for the case of the unit of Example~\ref{ref5}. As a consequence, we will show in Theorem~\ref{ref30} that $h_u$ is piecewise-smooth, with finitely many points of nondifferentiability which can be explicitly determined. In the case of the above unit, $C=1.14002\ldots$ and the points of nondifferentiability are
\[
\frac{3C}{\pi^2}\biggl\{\frac{1}{4},\frac{576}{625},1,\frac{1089}{625},
\frac{2304}{625},4,\frac{4356}{625}\biggr\}.
\]
The first, second, and fifth of these points (at $0.08663\ldots$, $0.31935\ldots$ and $1.27743\ldots$, respectively) 
are clearly visible in Figure~\ref{ref2}, while the others are quite hidden.

It is a pleasure to thank Jayadev Athreya for introducing me to the study of gap distribution via homogeneous dynamics, and for many clarifying and stimulating discussions on these topics.

\section{Basics}\label{ref24}

We first prove a fact which is interesting in its own right.

\begin{theorem}
There\label{ref36} exists $Q'$ such that, for every $Q\ge Q'$, all intervals $[p/q,p'/q']$ between successive elements of $\Fcal_u(Q)\cup\set{1}$ are unimodular.
\end{theorem}
\begin{proof}
Recall that the \newword{Ford circle} $C_{p/q}$ at the rational number $p/q$ is the circle of radius $1/(2q^2)$, lying in the upper-half plane and tangent to the real axis at $p/q$~\cite{ford38}.
The circles at the points of $\Fcal_u(Q)\cup\set1$ are then precisely those touching the real axis at points $s\in\ooii\cap\Qbb$ and having center on or above the graph of $u(s)^2/(2Q^2)$; see Figure~\ref{ref37} for the case of the unit of Example~\ref{ref5} and $Q=5$.
\begin{figure}[h!]
\includegraphics[height=4.5cm]{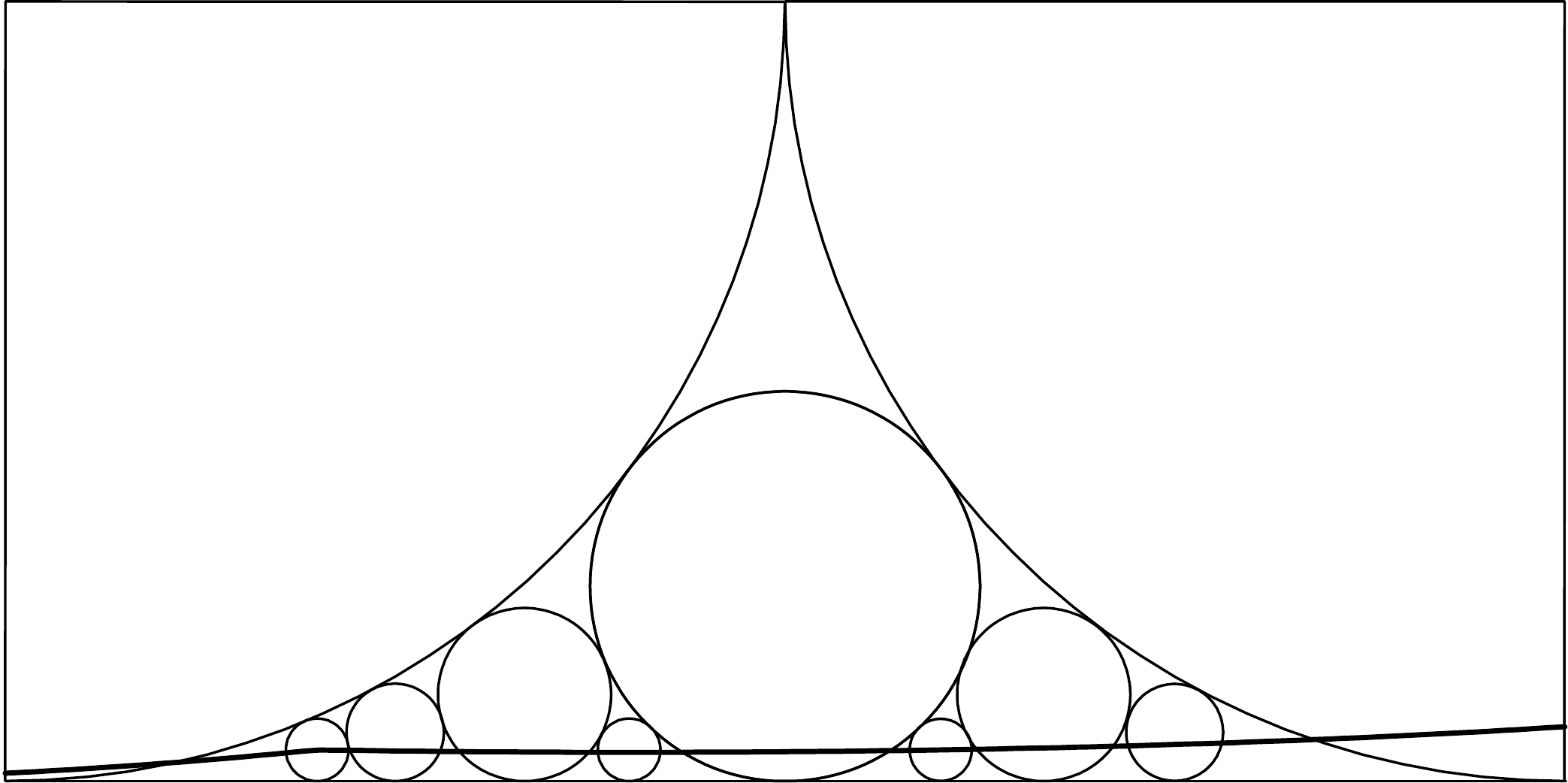}
\caption{}
\label{ref37}
\end{figure}
By the basic properties of the construction, two distinct circles $C_{p/q}$ and $C_{p'/q'}$ are either wholly external to one another, or tangent, and are tangent iff $[p/q,p'/q']$ is unimodular.

Suppose now that the statement of the theorem is false. Then there exists an infinite sequence $Q_0<Q_1<Q_2<\cdots$ such that, for every $i=0,1,2,\ldots$, some interval $[p(i)/q(i),p'(i)/q'(i)]$ between successive elements of $\Fcal_u(Q_i)\cup\set1$ is not unimodular. Without loss of generality  $q(i)\ge q'(i)$ and $q(i)>1$. By~\cite[Theorem~4]{ford38} there exists precisely one fraction $p''/q''$ such that $p(i)/q(i)<p''/q''$, $q(i)>q''$, and $[p(i)/q(i),p''/q'']$ is unimodular. The point $p''/q''$ must necessarily lie between $p(i)/q(i)$ and $p'(i)/q'(i)$, since otherwise $C_{p'(i)/q'(i)}$ and $C_{p''/q''}$ would intersect. Since $p(i)/q(i)$ and $p'(i)/q'(i)$ are consecutive in $\Fcal_u(Q_i)\cup\set1$, the rational $p''/q''$ does not belong to $\Fcal_u(Q_i)\cup\set1$, and therefore
\[
\frac{u(p''/q'')^2}{2Q_i^2}>\frac{1}{2(q'')^2}>
\frac{1}{2(q(i))^2}\ge\frac{u(p(i)/q(i))^2}{2Q_i^2}.
\]
Hence
\begin{multline*}
\frac{u(p''/q'')^2-u(p(i)/q(i))^2}{Q_i^2(p''/q''-p(i)/q(i))}>
\frac{(q'')^{-2}-(q(i))^{-2}}{p''/q''-p(i)/q(i)}=
\frac{q(i)^2-(q'')^2}{q''q(i)}\ge\\
\frac{q(i)^2-(q(i)-1)^2}{(q(i)-1)q(i)}=
\frac{2q(i)-1}{q(i)^2-q(i)}>\frac{1}{q(i)}.
\end{multline*}
As $p(i)/q(i)\in\Fcal_u(Q_i)$, we have $Q_i\ge u(p(i)/q(i))q(i)\ge
\min(u)\,q(i)$, and hence
\[
\frac{u(p''/q'')^2-u(p(i)/q(i))^2}{p''/q''-p(i)/q(i)}>
\frac{Q_i^2\min(u)}{Q_i}=Q_i\min(u).
\]
Now the last term tends to infinity, but this is impossible, since $u$ being a unit immediately implies that the set
\[
\biggl\{\biggl\lvert\frac{u(s)^2-u(s')^2}{s-s'}\biggr\rvert:
s\not= s'\in\ooii\biggr\}
\]
is bounded.
\end{proof}

For the rest of this paper, and without loss of generality, we assume that $Q$ is so large to satisfy the statement of Theorem~\ref{ref36}.

We define
\[
\nabla(1)=\biggl\{
\begin{pmatrix}
x\\ y
\end{pmatrix}
\in\Cone\ooi:
\bb{u}
\begin{pmatrix}
x\\ y
\end{pmatrix}
\le 1\biggr\};
\]
it constitutes a star-shaped sector, bounded by the lines $\set{x=0}$, $\set{x=y}$, and by the curve $\set{v(s)\cvect{s}{1}:0\le s<\le 1}$. 

\begin{lemma}
We\label{ref25} have:
\begin{itemize}
\item[(i)] the area of $\nabla(1)$ is $C/2$;
\item[(ii)] $n(Q)$ is asymptotic to $CQ^2/\bigl(2\,\zeta(2)\bigr)$;
\item[(iii)] as $Q\to\infty$, the probability on $\ooi$
\[
\frac{1}{n(Q)}\sum\set{\delta_s:s\in\Fcal_u(Q)}
\]
(where $\delta_s$ is the Dirac mass at $s$) converges weakly${}^*$ to $m(s)\ud s$.
\end{itemize}
\end{lemma}
\begin{proof}
(i) By elementary calculus, the sector swept by the line segment $\set{w\,v(s)\cvect{s}{1}:0<w\le 1}$ in time $\ud s$ has area $\ud A=2\m\,v(s)^2\ud s$; hence
$$
\area\bigl(\nabla(1)\bigr)=\int_0^1 2\m\,v(s)^2\ud s = C/2.
$$

(ii)-(iii) Fix a subinterval $[a,b)$ of $\ooi$. The cardinality of
$\Fcal_u(Q)\cap [a,b)$
is equal to the number of primitive points of the lattice $Q\m\cvect{\Zbb}{\Zbb}$ inside the sector $\set{w\,v(s)\cvect{s}{1}:0<w\le1\text{ and }a\le s< b}$. For $Q\to\infty$, the number of such points is asymptotic to $Q^2$ times the area of the sector divided by $\zeta(2)$~\cite[Theorem~459]{hardywri85}, and it follows that the ratio between the number of primitive points inside the sector and the total number of primitive points in $\nabla(1)$ is asymptotic to the ratio of the relative areas. By the first part of the proof we then have
\[
\lim_{Q\to\infty}\frac{\sharp\bigl(\Fcal_u(Q)\cap [a,b)\bigr)}{n(Q)}=
2/C\int_a^b 2\m v(s)^2\ud s=
\int_a^b m(s)\ud s.
\]
\end{proof}

We recall a few basic facts about the horocycle flow; see~\cite[Chapter~IV]{bekkamayer} or~\cite[Chapter~11]{einsiedlerward} for a full treatment. The group $\PSL_2\Rbb$ acts on the upper halfplane $\Hcal\subset\Cbb$ on the left: if
$A=\smatrix abcd$ and $\alpha\in\Hcal$, then
\[
\begin{pmatrix}
a & b \\
c & d
\end{pmatrix}
\ast\alpha
=\frac{a\alpha+b}{c\alpha+d}.
\]
We identify the unit tangent space at $\alpha$ with $\set{\tau\in\Cbb:\abs{\tau}=\im\alpha}$; the above action extends then to a left action of $\PSL_2\Rbb$ on the unit tangent bundle $T^1\Hcal$ via $A\ast(\alpha,\tau)=(A\ast\alpha,(c\alpha+d)^{-2}\tau)$. This latter action is transitive with trivial stabilizers, so we get a bijection
\[
\PSL_2\Rbb\ni A \mapsto A\ast (i,i)\in T^1\Hcal,
\]
whose inverse is given by the Iwasawa decomposition
\[
\begin{pmatrix}
1 & \re\alpha \\
0 & 1
\end{pmatrix}
\begin{pmatrix}
(\im\alpha)^{1/2} & 0 \\
0 & (\im\alpha)^{-1/2}
\end{pmatrix}
\begin{pmatrix}
\cos\theta & -\sin\theta \\
\sin\theta & \cos\theta
\end{pmatrix}
\leftarrow\!\mapstochar
(\alpha,\tau),
\]
where $\theta=-\arg(\tau/i)/2$.

The above bijection is preserved under the quotient by the left action of the modular group $\Gamma=\PSL_2\Zbb$, so we have an induced bijection
\begin{equation*}
\Phi:\Gamma\bs\PSL_2\Rbb \to T^1M,
\end{equation*}
where $M=\Gamma\bs\Hcal$ is the modular surface. We keep on writing $\xi=(\alpha,\tau)$ for the points of $T^1M$, without explicit mention of the action of~$\Gamma$.

We can extend $\Phi$ to the space $X_2$ of all unimodular lattices in $\Rbb^2$. To this purpose, it is expedient to identify $\Rbb^2$ with $\Cbb$ as real vector spaces, with the standard basis $\rvect{\bb{e}_1}{\bb{e}_2}$ corresponding to $\rvect{i}{1}$ (we always drop reference to $\rvect{\bb{e}_1}{\bb{e}_2}$, by writing $\cvect{x}{y}$ for $\rvect{\bb{e}_1}{\bb{e}_2}\cvect{x}{y}$).
We also set $S=\smatrix{0}{-1}{1}{0}$ and remark that, for $A\in\PSL_2\Rbb$, the matrix $SA\m S\m$ is the transpose $A^T$ of $A$.

\begin{lemma}
Let\label{ref16} $\Xi:\Gamma\bs\PSL_2\Rbb\to X_2$ and $\Psi:T^1M\to X_2$ be defined by
\begin{align*}
\Xi:\Gamma A & \mapsto 
SA\m S\m
\begin{pmatrix}
\Zbb \\
\Zbb
\end{pmatrix}
=A^T
\begin{pmatrix}
\Zbb \\
\Zbb
\end{pmatrix},\\
\Psi:(\alpha,\tau) & \mapsto 
(i/\tau)^{1/2}
\begin{pmatrix}
\alpha & 1
\end{pmatrix}
\begin{pmatrix}
\Zbb \\
\Zbb
\end{pmatrix}.
\end{align*}
Then $\Xi,\Psi$ are well-defined bijections and the square
\[
\begin{xy}
\xymatrix{
\Gamma\bs\PSL_2\Rbb \ar[r]^\Phi \ar[d]_\Xi &
T^1M \ar[d]^\Psi \\
X_2 \ar@{=}[r] & X_2
}
\end{xy}
\]
whose bottom row results from the identification of $\Rbb^2$ with $\Cbb$ as above, commutes.
\end{lemma}
\begin{proof}
The definition of $\Xi$ is clearly unambiguous, and so is that of $\Psi$, since $(i/\tau)^{1/2}$ is determined up to sign. One checks easily that $\Xi$ and $\Psi$ are bijections and that $\Xi\m\Psi\Phi=\id$.
\end{proof}

The group $\PSL_2\Rbb$ acts on the right on all spaces in Lemma~\ref{ref16}. Namely, for $A,B,R\in\PSL_2\Rbb$ we have
\begin{align*}
\Gamma A\ast R &= \Gamma AR,\\
\bigl(A\ast(i,i)\bigr)\ast R &= AR\ast(i,i),\\
B\begin{pmatrix}
\Zbb\\
\Zbb
\end{pmatrix}\ast R &= R^T B
\begin{pmatrix}
\Zbb\\
\Zbb
\end{pmatrix},\\
\begin{pmatrix}
i & 1
\end{pmatrix}
B\begin{pmatrix}
\Zbb\\
\Zbb
\end{pmatrix}\ast R &= 
\begin{pmatrix}
i & 1
\end{pmatrix}
R^T B
\begin{pmatrix}
\Zbb\\
\Zbb
\end{pmatrix}.
\end{align*}
By construction, the square in Lemma~\ref{ref16} is equivariant under this action. In particular, the (unstable, time-reversing) horocycle flow $h_t$ is defined on $T^1M$ and~$X_2$ by
\begin{align*}
h_t:A\ast(i,i)&\mapsto A
\begin{pmatrix}
1 & 0\\
-t & 1
\end{pmatrix}
\ast(i,i),\\
h_t:B\begin{pmatrix}
\Zbb\\
\Zbb
\end{pmatrix}
&\mapsto
\begin{pmatrix}
1 & -t\\
0 & 1
\end{pmatrix}
B\begin{pmatrix}
\Zbb\\
\Zbb
\end{pmatrix}.
\end{align*}

\section{A sliding section for the horocycle flow}\label{ref21}

We say that the lattice $\Lambda\in X_2$ \newword{contains a vertical vector} if it contains a vector of the form $\cvect{0}{d}$ for some $d>0$.
The set
\[
\Scal=\set{\Lambda\in X_2:\Lambda\text{ contains a vertical vector}}
\]
is a 2-dimensional immersed submanifold of $X_2$, dense in $X_2$.
We can easily parametrize $\Scal$: indeed, every $\Lambda\in\Scal$ is of the form
\[
\Lambda=\begin{pmatrix}
d\m & 0\\
c & d
\end{pmatrix}
\begin{pmatrix}
\Zbb \\ \Zbb
\end{pmatrix},
\]
with $d>0$ uniquely determined, and $c$ determined up to translation by integer multiples of $d$.
This can be described as follows: let $\Zbb$ act on the real upper halfplane $\Rbb\times\Rbb\pp$ by $k\ast(c,d)=(c+kd,d)$. Then the map
\begin{align*}
\tilde\phi:\Scal&\to \Zbb\backslash(\Rbb\times\Rbb\pp),\\
\Lambda&\mapsto 
\Zbb\ast(c,d),
\end{align*}
is a homeomorphism, so that $\Scal$ is an immersed cylinder.

For every $D>0$, the set $\Scal\restriction D$ of all elements of $\Scal$ that contain a vertical vector of length $\le D$ is a Poincar\'e section for the horocycle flow. Indeed, all elements of $X_2$ ---except the codimension-$1$ set of lattices of the form $\smatrix{a}{b}{0}{a\m}\cvect{\Zbb}{\Zbb}$ for some $a<D\m$--- enter $\Scal\restriction D$ countably many times under the action of the flow, both in the past and in the future.
Up to the bijection $\Phi\circ\Xi\m$ the cylinder
$\Scal$ is identifiable with $\Gamma\backslash\set{(\alpha,(\im\alpha)\,i):\alpha\in\Hcal}$. The section
$\Scal\restriction D$ is then the set of all elements of $T^1M$ that have a lift to $T^1\Hcal$ of the form $\bigl(\alpha,(\im\alpha)\,i\bigr)$ for some $\alpha$ having imaginary part $\ge D^{-2}$.

We now want ---this being the key idea in this paper--- to let $D$ vary with time.

\begin{definition}
Given\label{ref14} $Q$, for $t$ in $[0,Q^2)$ we define $\Scal_t=\Scal\restriction v(Q^{-2}t)$.
Let $\Lambda_0(Q)=\smatrix{Q}{0}{0}{Q\m}\cvect{\Zbb}{\Zbb}$; we safely assume $Q\m\le v(0)$, so that $\Lambda_0(Q)\in\Scal_0$. Writing $\Lambda_t(Q)$ for $h_t\bigl(\Lambda_0(Q)\bigr)$, we say that $\Lambda_0(Q)$ \newword{hits} the sliding section $\Scal_t$ at time~$t$ if $\Lambda_t(Q)\in\Scal_t$.
\end{definition}

For ease of notation, whenever $Q$ is understood we write $\Lambda_t$ for $\Lambda_t(Q)$.

\begin{lemma}
The\label{ref26} $h_t$-orbit of $\Lambda_0$ is periodic of period $Q^2$.
The hitting times $0=t_0<t_1<t_2<\cdots$ are precisely the multiples
$t_i=Q^2s_i$ of the elements $s_i\in\Fcal_u(Q)$.
\end{lemma}
\begin{proof}
The first statement is clear.
As remarked in the proof of Lemma~\ref{ref25}, the elements of $\Fcal_u(Q)$ are in 1-1 correspondence
with the primitive points of the lattice $Q\m\cvect{\Zbb}{\Zbb}$ which are inside~$\nabla(1)$.
Let
\[
\nabla(Q)=\biggl\{
\begin{pmatrix}
Q^2x\\
y
\end{pmatrix}:
\begin{pmatrix}
x\\
y
\end{pmatrix}
\in\nabla(1)\biggr\};
\]
then
\[
\frac{p}{q}\leftrightarrow
\begin{pmatrix}
Qp\\ Q\m q
\end{pmatrix}
\]
is a 1-1 correspondence between the points of 
$\Fcal_u(Q)$ and the primitive points of $\Lambda_0$
that are in $\nabla(Q)$.

Now, to say that $\Lambda_t\in\Scal_t$ amounts to saying that the ray $\Rbb\pp\cvect t1$ passes through one of these primitive points. Therefore $t$ is a hitting time iff $Qp=tQ\m q$ (i.e., $t=Q^2p/q$) for some $p/q\in\Fcal_u(Q)$.
\end{proof}

Let $t_i=Q^2s_i=Q^2p_i/q_i$ be a hitting time. We choose a lift $\phi(\Lambda_{t_i})$ of $\tilde\phi(\Lambda_{t_i})=\Zbb\ast(c,d_i)$ to $\Rbb\pp^2$ as follows: of all possible choices for $c$ we pick the largest one --- call it $c_i$---
such that
\[
0< c \le v\bigl(Q^{-2}(t_i+(cd_i)\m)\bigr),
\]
and set $\phi(\Lambda_{t_i})=(c_i,d_i)$.

\begin{lemma}
We\label{ref22} have
\begin{align*}
(c_i, d_i)&=Q\m(q_{i+1},q_i),\\
t_{i+1}-t_i&=\frac{1}{c_id_i}.
\end{align*}
\end{lemma}
\begin{proof}
By construction, and according to Lemma~\ref{ref26},
\[
h_{t_i}\m(\Lambda_{t_i})=\Lambda_0=
\begin{pmatrix}
Qp_{i+1} & Qp_i\\
Q\m q_{i+1} & Q\m q_i
\end{pmatrix}
\begin{pmatrix}
\Zbb\\ \Zbb
\end{pmatrix};
\]
this is justified by our standing assumption that $Q$ is so large that all intervals determined by $\Fcal_u(Q)$ are unimodular, as guaranteed by Theorem~\ref{ref36}. We thus get
\[
\Lambda_{t_i}=\begin{pmatrix}
1 & -Q^2p_i/q_i\\
0 & 1
\end{pmatrix}
\begin{pmatrix}
Qp_{i+1} & Qp_i\\
Q\m q_{i+1} & Q\m q_i
\end{pmatrix}
\begin{pmatrix}
\Zbb\\ \Zbb
\end{pmatrix}=
\begin{pmatrix}
Qq_i\m & 0\\
Q\m q_{i+1} & Q\m q_i
\end{pmatrix}
\begin{pmatrix}
\Zbb\\ \Zbb
\end{pmatrix},
\]
so that $d_i=Q\m q_i$.
Let
\[
s'=\frac{p'}{q'}=\frac{p_{i+1}+p_i}{q_{i+1}+q_i}
\]
be the Farey mediant of $s_i$ and $s_{i+1}$. By definition of $\Fcal_u(Q)$, we have $u(s_{i+1})q_{i+1}\le Q$ and $u(s')q'>Q$.
Therefore
\[
0<Q\m q_{i+1}\le
v(s_{i+1})
=v\bigl(Q^{-2}t_i+(q_{i+1}q_i)\m\bigr)=
v\bigl(Q^{-2}(t_i+(Q\m q_{i+1}d_i)\m)\bigr),
\]
while
\begin{multline*}
Q\m q_{i+1}+d_i=Q\m q'>
v(s')
=v\bigl(s_i+(q'q_i)\m\bigr)\\
=v\bigl(Q^{-2}(t_i+(Q\m q' Q\m q_i)\m)\bigr)=
v\bigr(Q^{-2}(t_i+((Q\m q_{i+1}+d_i)d_i)\m)\bigr).
\end{multline*}
Therefore $c_i=Q\m q_{i+1}$ as claimed, and $t_{i+1}-t_i=Q^2(s_{i+1}-s_i)=Q^2(q_{i+1}q_i)\m=(c_id_i)\m$.
\end{proof}

\begin{definition}
We\label{ref27} denote the minimum and maximum of $v$ on $\ooii$
by $l$ and $L$, respectively.
For $w\in[l,L]$ we set:
\begin{align*}
\Omega_w &=\text{the triangle $\set{(x,y)\in\Rbb\pp^2:x,y\le w<x+y}$},\\
P_w&=\text{the Lebesgue measure on $\Omega_w$, normalized by $P_w(\Omega_w)=1$},\\
\Omega&=\textstyle{\bigcup_w}\Omega_w=\text{the pentagon }\set{0<x,y\le L}\cap\set{l<x+y},\\
P&=\text{the probability measure on $\Omega$ defined by }\int_0^1 P_{v(s)}\,m(s)\ud s.
\end{align*}
Given $Q$, we write $\phi(\Lambda_{t_i})=(c_i,d_i)$ for any hitting time $t_i$ of $\Lambda_0$ to the sliding section $\Scal_t$.
Then
\[
P(Q)=\frac{1}{n(Q)}\sum\delta_{(c_i,d_i)}
\]
is a point-process probability measure
on $\Rbb\pp^2$.
\end{definition}

The following is our main theorem.

\begin{theorem}
As\label{ref19} $Q$ goes to infinity, $P(Q)$ converges weakly${}^*$ to $P$.
\end{theorem}

We will prove Theorem~\ref{ref19} in \S\ref{ref23}. In Figure~\ref{ref29} we plot the support of $P(400)$ for the unit of Example~\ref{ref5}.
\begin{figure}[h!]
\includegraphics[width=8cm]{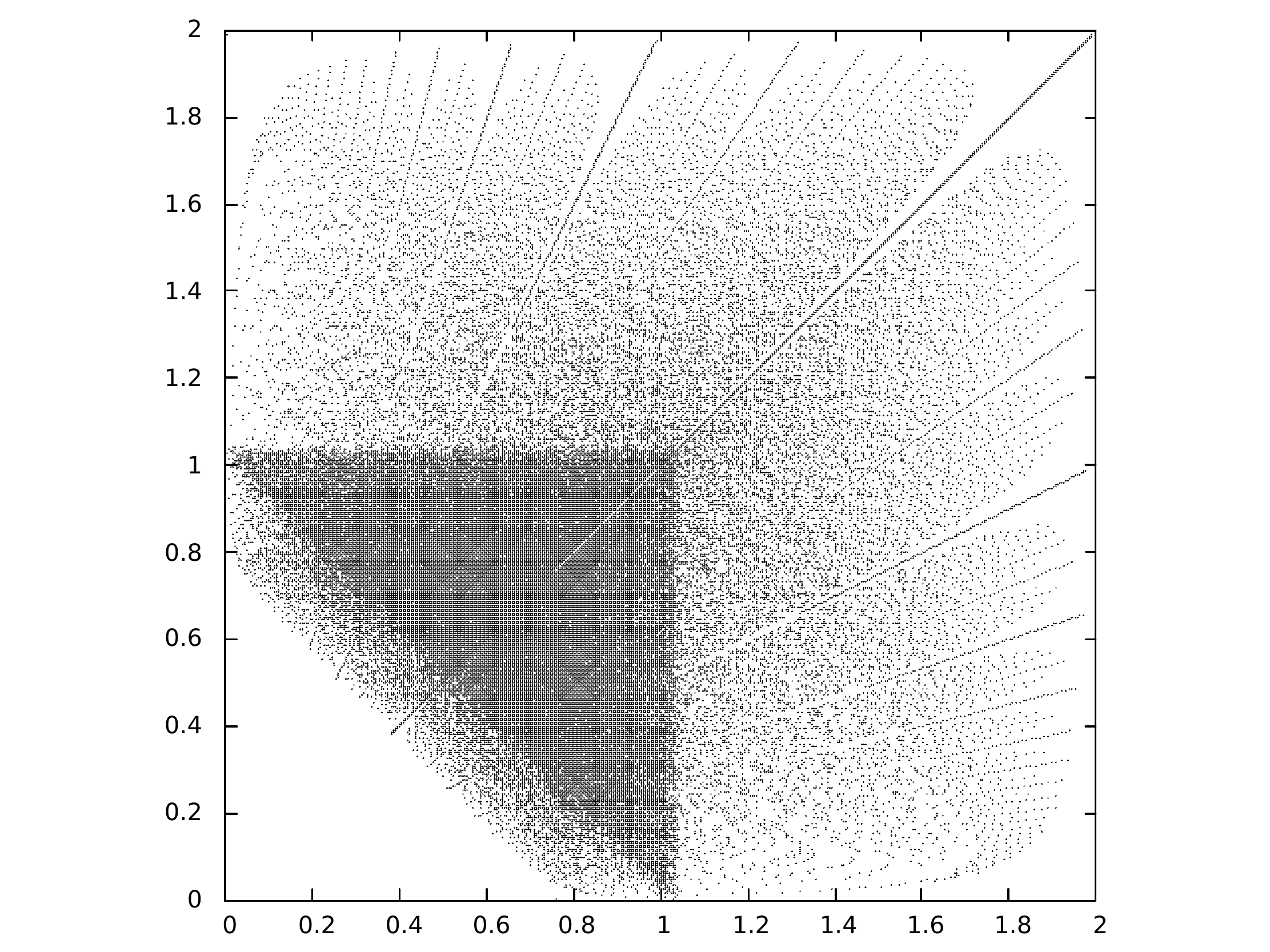}
\caption{}
\label{ref29}
\end{figure}

\begin{theorem}
The\label{ref20} gap distribution $H_u(z)$ of the weighted Farey sequence $\Fcal_u(Q)$
is the cumulative density distribution of the random variable
\begin{align*}
Z:(\Omega,P) &\to\Rbb\pp,\\
(x,y) &\mapsto \frac{C}{2\,\zeta(2)\,xy}.
\end{align*}
Also,
\[
P(Z\le z)=\int_0^1 H_\one\bigl(m(s)\, z\bigr)\, m(s)\ud s,
\]
so formula~\eqref{eq2} holds.
\end{theorem}
\begin{proof}
Fix $z>0$. Then
\begin{align*}
&\lim_{Q\to\infty}\frac{1}{n(Q)}
\sharp\set{0\le i<n(Q):\nngg_Q(s_i)\le z}\\
=&
\lim_{Q\to\infty}\frac{1}{n(Q)}
\sharp\set{i:\bigl(2\,\zeta(2)\bigr)\m CQ^2(s_{i+1}-s_i)\le z}&&\text{by Lemma~\ref{ref25}(ii)}\\
=&\lim_{Q\to\infty}\frac{1}{n(Q)}
\sharp\set{i:\bigl(2\,\zeta(2)\bigr)\m C(c_id_i)\m\le z} &&\text{by Lemma~\ref{ref22}}\\
=&\lim_{Q\to\infty}\bigl(P(Q)\bigr)(Z\le z)\\
=&\;P(Z\le z)&&\text{by Theorem~\ref{ref19}}.
\end{align*}

For the second claim we observe that $Z(x,y)\le z$ iff $Z(w\m x,w\m y)\le w^2z$, so that $P_w(Z\le z)=P_1(Z\le w^2z)$ for every $w\in[l,L]$. We then compute
\begin{align*}
P(Z\le z)&=\int_0^1 P_{v(s)}(Z\le z)\,m(s) \ud s\\
&=\int_0^1 P_1(Z\le v(s)^2\,z)\,m(s) \ud s\\
&=\int_0^1 P_1(Z_1\le C\m v(s)^2\,z)\,m(s) \ud s\\
&=\int_0^1 H_\one(m(s)\, z)\, m(s) \ud s.
\end{align*}
\end{proof}

From here on $\lambda$ denotes the $1$-dimensional Lebesgue measure.

\begin{theorem}
(i) 
The\label{ref30} measure $P$ is absolutely continuous w.r.t.~$\ud x\wedge\ud y$. More precisely,
let
\[
p(x,y)=2C\m(v_*\lambda)\bigl([\max(x,y),x+y)\cap[l,L]\bigr),
\]
where $v_*\lambda$ is the pushforward of $\lambda$ by $v$.
Then
\[
\ud P=p(x,y)\ud x\wedge\ud y.
\]

(ii) The probability density function $h_u$ is piecewise-smooth, with finitely many nondifferentiability points. These points are those in the set
\[
\frac{C}{2\,\zeta(2)}\bigl(\set{u(s)^2:s\in E}\cup\set{4u(s)^2:s\in E}\bigr),
\]
where $E\subset\ooii$ contains $0$, $1$, and all points at which $v$ (equivalently, $u$) is nondifferentiable or has a local maximum or minimum.
\end{theorem}
\begin{proof}
(i) By definition, $\ud P_w=2w^{-2}\one_{\Omega_w}(x,y)\ud x\wedge\ud y$. Thus, by Fubini,
\begin{align*}
\int f\ud P &= \int_0^1\biggl[\int f\ud P_{v(s)}\biggr]\,m(s)\ud s\\
&=\int_0^1\biggl[\int f\,2\,v(s)^{-2}\one_{\Omega_{v(s)}}(x,y)
\ud x\wedge\ud y\biggr]\,C\m v(s)^2\ud s\\
&=2C\m \int f\biggl[\int_0^1 \one_{\Omega_{v(s)}}(x,y)\ud s\biggr]\ud x\wedge\ud y\\
&=2C\m \int f\,\lambda\set{s:x,y\le v(s)<x+y}\ud x\wedge\ud y\\
&=2C\m \int f\,\lambda(v\m[\max(x,y),x+y))\ud x\wedge\ud y\\
&=\int f\,p(x,y)\ud x\wedge\ud y.
\end{align*}

(ii) This is best conveyed in geometrical language. As the sliding triangle $\Omega_{v(s)}$ moves through $\Omega$, it deposits mass. By our assumptions about $u$, this process is smooth except at the points $s\in E$, where the triangle starts or stops moving, reverses direction, or changes speed abruptly. We thus get a finite set $\set{\Omega_{v(s)}:s\in E}$ of triangles along whose borders the density $p(x,y)$ is singular. Now, the probability distribution function $h_u$ is singular at $z$ precisely when the hyperbola $\set{Z(x,y)=z}$ ---in its downward movement as $z$ goes from $0$ to infinity--- touches one of these triangles, say $\Omega_{v(s)}$, either in the upper right corner or in the midpoint of the hypothenuse. In the first case we have $Z(v(s),v(s))=z$ (i.e., $z=\bigl(2\,\zeta(2)\bigr)\m Cu(s)^2$), while in the second we have $Z(v(s)/2,v(s)/2)=z$ (i.e., $z=2\,\zeta(2)\m Cu(s)^2$).
\end{proof}

\section{Proof of Theorem~\ref{ref19}}\label{ref23}

Fix a sequence of real numbers $\varepsilon_0,\varepsilon_1,\varepsilon_2,\ldots$, strictly decreasing and converging to $0$. For every $k$, fix $Q_k$ so large that every $v$-image of an interval in $\Fcal_u(Q_k)$ has length $\le\varepsilon_k$. Given $k$ and $0\le h<n(Q_k)$, we let $I(k,h)$ be the $h$-th interval
(closed to the left and open to the right) of the partition of $\ooi$ determined by $\Fcal_u(Q_k)$. We also let $r(k,h)$ be a point in the topological interior of $I(k,h)$ such that
\begin{equation}\label{eq5}
\lim_{Q\to\infty}
\frac{\sharp\bigl(\Fcal_u(Q)\cap I(k,h)\bigr)}{\sharp\Fcal_u(Q)}=
\int_{I(k,h)} m(s)\ud s=
m(r(k,h))\,\lambda(I(k,h));
\end{equation}
such a point exists by Lemma~\ref{ref25}(iii) and the intermediate value theorem for integrals of continuous functions.

For short, we write
\begin{align*}
\Fcal_u(Q,k,h)&=\Fcal_u(Q)\cap I(k,h),\\
\Fcal_{u(r)}(Q,k,h)&=\Fcal_{u(r(k,h))}(Q)\cap I(k,h).
\end{align*}
Note that $u(r(k,h))$ and $v(r(k,h))$ may denote either a number or ---as above--- the constant function whose value is that number; the context always makes the meaning clear.

\begin{lemma}
For\label{ref15} $Q\to\infty$, the cardinalities of $\Fcal_u(Q,k,h)$ and of $\Fcal_{u(r)}(Q,k,h)$ are asymptotically equal.
\end{lemma}
\begin{proof}
As in Lemma~\ref{ref25}, the ratio of the two cardinalities is asymptotic to the ratio of the areas of the two sectors
$\set{w\,v(s)\cvect{s}{1}:0<w\le 1\text{ and }s\in\ I(k,h)}$
and $\set{w\,v(r(k,h))\cvect{s}{1}:0<w\le 1\text{ and }s\in\ I(k,h)}$.
The first sector has area
$$
\int_{I(k,h)} 2\m v(s)^2\ud s,
$$
while the second has area
$$
2\m  v(r(k,h))\bigl(\lambda(I(k,h))\,v(r(k,h))\bigr).
$$
Taking into account the definition of $r(k,h)$ and the fact that $v(s)^2=C\,m(s)$, one checks immediately that the two areas agree.
\end{proof}

Let $l(k,h)$, $L(k,h)$ be the infimum and the supremum of $v$ on $I(k,h)$, respectively. Then the triangle
$$
\Omega_{I(k,h)}^o=\set{x,y\le l(k,h)}\cap\set{L(k,h)< x+y}
$$
lies inside $\Omega_{v(r(k,h))}$. In turn, the latter lies inside the pentagon
$$
\Omega_{I(k,h)}=\set{0<x,y\le L(k,h)}\cap\set{l(k,h) < x+y};
$$
see Figure~\ref{ref33}. 
\begin{figure}[h!]
\includegraphics[width=8cm]{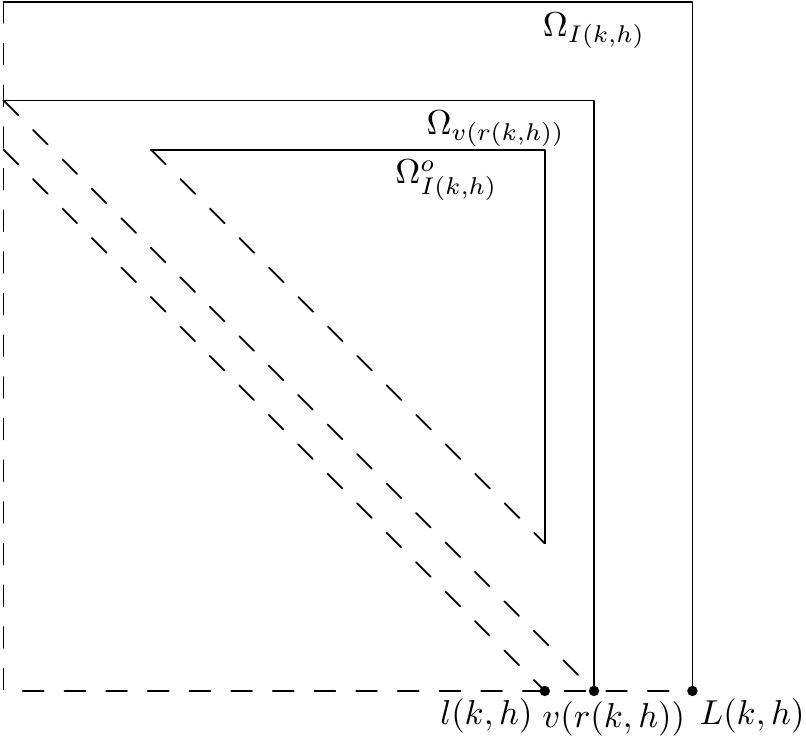}
\caption{}
\label{ref33}
\end{figure}

For $Q\ge Q_k$ and $t$ varying in $Q^2I(k,h)$, we are aiming at freezing the sliding section $\Scal_t$ to the fixed section $\Scal_{Q^2r(k,h)}$. Lemma~\ref{ref15} guarantees that for large $Q$'s this has no impact on the number of hits.
The following Lemma~\ref{ref34} copes with the change of the coordinate function. Indeed, replacing $u$ with the constant function $u(r(k,h))$ we are forced to assign the lattice
\begin{equation}
\label{ref18}
\Lambda_t(Q)=
\begin{pmatrix}
d\m & 0\\
c & d
\end{pmatrix}
\begin{pmatrix}
\Zbb \\ \Zbb
\end{pmatrix}
\end{equation}
(with $t\in Q^2I(k,h)$ and $d\le v(r(k,h))$) the coordinates $\phi'(\Lambda_t(Q))=(c',d)$, where $c'$ is the largest element in $c+d\Zbb$ such that $c'\le v(r(k,h))$; assuming that $t$ is a hitting time for the sliding section as well, the two lifts, $\phi(\Lambda_t)$ and $\phi'(\Lambda_t)$, of $\tilde\phi(\Lambda_t)=\Zbb\ast(c,d)$ may differ in their first component.

\begin{lemma}
Let\label{ref34} $Q\ge Q_k$, $t\in Q^2I(k,h)$, and let $\Lambda_t=\Lambda_t(Q)$ be as in~\eqref{ref18}. Then:
\begin{itemize}
\item[(i)] if $t$ is a hitting time for the sliding section, then $\phi(\Lambda_t)\in\Omega_{I(k,h)}$;
\item[(ii)] if $t$ is a hitting time for the fixed section, then $\phi'(\Lambda_t)\in\Omega_{v(r(k,h))}$;
\item[(iii)] if $\Sigma\ast(c,d)=\Sigma\ast(c',d)$ with $(c',d)\in\Omega_{I(k,h)}^o$, then $t$ is a hitting time for both sections and $\phi(\Lambda_t)=\phi'(\Lambda_t)=(c',d)$.
\end{itemize}
\end{lemma}
\begin{proof}
(i) Say $Q^{-2}t=s_i\in I(k,h)$, with $s_{i+1}$ the element following $s_i$ in $\Fcal_u(Q)$. By Lemma~\ref{ref22},
$\phi(\Lambda_t)=(Q\m q_{i+1},Q\m q_i)$. Since $Q\m q_i\le v(s_i)$ and $Q\m q_{i+1}\le v(s_{i+1})$, both $Q\m q_i$ and $Q\m q_{i+1}$ are $\le L(k,h)$. On the other hand, as the Farey mediant $s'$ of $s_i$ and $s_{i+1}$ does not belong to $\Fcal_u(Q)$, we have $Q\m q_{i+1}+Q\m q_i=Q\m(q_{i+1}+q_i)>v(s')\ge l(k,h)$.
(ii) follows from (i), applied to the constant unit $u(r(k,h))$.
(iii) Assume $(c',d)\in\Omega^o_{I(k,h)}$. Then $d\le l(k,h)\le v(Q^{-2}t), v(r(k,h))$, so $t$ is a hitting time for both sections. Since $c'+d>L(k,h)$ and $(c'-d)+d=c'\le l(k,h)$, the only lift of $\tilde\phi(\Lambda_t(Q))$ which is contained in $\Omega_{I(k,h)}$ is $(c',d)$. By~(i) and~(ii) both lifts $\phi(\Lambda_t)$ and $\phi'(\Lambda_t)$ must necessarily be equal to $(c',d)$.
\end{proof}

Recall from~\S\ref{ref21} that $\Scal\restriction v(r(k,h))$
is a transverse section for the horocycle flow. The Liouville measure on $X_2$ ---namely, the pushforward via $\Xi$ of the unique $\PSL_2\Rbb$ right invariant measure on $\Gamma\backslash\PSL_2\Rbb$, see the diagram in Lemma~\ref{ref16}--- decomposes locally as a product of a transverse measure for the horocycle foliation (the foliation whose leaves are the $h_t$-orbits) and the linear measure on the leaves.
Let $F$ be the map $\Phi\circ\Xi\m\circ\phi\m$ from $\Omega_{v(r(k,h))}$ to $T^1M$; then
$$
F(x,y)=(xy\m+y\md i,y\md i)=(\xi +\nu i,\nu i).
$$
Write $U$ for the infinitesimal generator of the horocycle flow, and $\omega=\nu\md\ud\xi\wedge\ud\nu\wedge\ud\theta$ for the riemannian volume form.
Then the transverse measure corresponds on 
$\Omega_{v(r(k,h))}$ to the pullback via $F$ of the contraction $\iota_U\omega$ of $\omega$ w.r.t.~the vector field~$U$. A straightforward computation shows now that $F^*(\iota_U\omega)=2\ud x\wedge\ud y$.

A key result by Sarnak~\cite{sarnak81} states that closed horocycles on $X_2$ become equidistributed as their length goes to infinity. This has been refined by Hejhal~\cite{hejhal96}, who shows that equidistribution still holds if the orbits are restricted to a constant fraction of their full length. In our case, this implies that for every pair $(k,h)$ the probability
$$
\frac{1}{\sharp\Fcal_{u(r)}(Q,h,k)}
\sum\set{\delta_{\phi'\Lambda_{Q^2s}}:s\in\Fcal_{u(r)}(Q,k,h)}
$$
converges weakly${}^*$ to the normalization of $2\ud x\wedge\ud y$ to a probability on $\Omega_{v(r(k,h))}$, namely to $\ud P_{v(r(k,h)}$. Summing up,
\begin{equation}\label{eq7}
\lim_{Q\to\infty}\frac{\sum\set{f(\phi'(\Lambda_{Q^2s}(Q))):s\in\Fcal_{u(r)}(Q,k,h)}}{\sharp\Fcal_{u(r)}(Q,k,h)}
=\int f\ud P_{v(r(k,h))},
\end{equation}
for every continuous function $f:\Rbb\pp^2\to\Rbb$ of compact support.

Given such a function $f$, let us write for short
\begin{align*}
\phi\Lambda_{Q^2s}&=\phi\bigl(\Lambda_{Q^2s}(Q)\bigr),\text{ and analogously for $\phi'$},\\
\alpha(Q,k,h)&=
\frac{\Sigma\set{f(\phi'\Lambda_{Q^2s}):s\in\Fcal_{u(r)}(Q,k,h)}-\set{f(\phi\Lambda_{Q^2s}):s\in\Fcal_u(Q,k,h)}}{\sharp\Fcal_{u(r)}(Q,k,h)},\\
a(Q,k)&=\sum_{h=0}^{n(Q_k)-1}\Biggl(\alpha(Q,k,h)\,\frac{\sharp\Fcal_u(Q,k,h)}{\sharp\Fcal_u(Q)}\Biggr).
\end{align*}

\begin{lemma}
Given\label{ref31} $k$ there exists $N(k)$ such that, for every $Q$ sufficiently large, both $\abs{\alpha(Q,k,h)}$ and $\abs{a(Q,k)}$ are $\le N(k)$. Also, $\lim_{k\to\infty}N(k)=0$.
\end{lemma}
\begin{proof}
Take first $R=\max(f)-\min(f)$. By Lemma~\ref{ref34}(iii) $\abs{\alpha(Q,k,h)}$ is bounded from above by
\begin{multline*}
R\Biggl(
\frac{\sharp\set{s\in\Fcal_{u(r)}(Q,k,h):\phi'\Lambda_{Q^2s}\notin\Omega_{I(k,h)}^o}}
{\sharp\Fcal_{u(r)}(Q,k,h)}\\
+
\frac{\sharp\set{s\in\Fcal_u(Q,k,h):\phi\Lambda_{Q^2s}\notin\Omega_{I(k,h)}^o}}
{\sharp\Fcal_{u(r)}(Q,k,h)}
\Biggr).
\end{multline*}
Possibly increasing $R$ to take care of the asymptotic $\sharp\Fcal_{u(r)}(Q,k,h)\sim\sharp\Fcal_u(Q,k,h)$ given by Lemma~\ref{ref15}, the above sum is bounded by
\[
2R
\frac{\sharp\set{s\in\Fcal_{u(r)}(Q,k,h):\phi'\Lambda_{Q^2s}\notin\Omega_{I(k,h)}^o}}
{\sharp\Fcal_{u(r)}(Q,k,h)}.
\]
By~\eqref{eq7} this last expression tends, for $Q\to\infty$, to
\[
2R\frac{\area\bigl(\Omega_{v(r(k,h))}\setminus\Omega_{I(k,h)}^o\bigr)}{\area\bigl(\Omega_{v(r(k,h))}\bigr)}.
\]
By elementary geometric considerations, the area ratio is bounded by $4\varepsilon_k$. Taking $N(k)=8R\varepsilon_k$ gives the bound on $\abs{\alpha(Q,k,h)}$, and the bound on $\abs{a(Q,k)}$ follows easily.
\end{proof}

We can now conclude the proof of Theorem~\ref{ref19}. Since $P$ is defined via a Riemann integral, we have
\begin{align*}
\int f \ud P &= \int_0^1\biggl[\int f\ud P_{v(s)}\biggr] m(s)\ud s\\
&=\lim_{k\to\infty}\sum_{h=0}^{n(Q_k)-1}
\Biggl(\biggl[\int f\ud P_{v(r(k,h))}\biggr]\,\lambda(I(k,h))\,m(r(k,h))\Biggr)\\
&=\lim_{k\to\infty}\sum_{h=0}^{n(Q_k)-1}\lim_{Q\to\infty}
\Biggl(\frac{\sum\set{f(\phi'\Lambda_{Q^2s}):s\in\Fcal_{u(r)}(Q,k,h)}}{\sharp\Fcal_{u(r)}(Q,k,h)}
\cdot
\frac{\sharp\Fcal_u(Q,k,h)}{\sharp\Fcal_u(Q)}\Biggr).
\end{align*}
Summing and subtracting
$$
\frac{\sum\set{f(\phi\Lambda_{Q^2s}):s\in\Fcal_u(Q,k,h)}}{\sharp\Fcal_{u(r)}(Q,k,h)}
$$
to the left factor of the parenthesized product above, we get
\begin{multline*}
\int f \ud P =\\
\lim_{k\to\infty}\sum_{h=0}^{n(Q_k)-1}\lim_{Q\to\infty}
\Biggl(
\frac{\sum\set{f(\phi\Lambda_{Q^2s}):s\in\Fcal_u(Q,k,h)}}{\sharp\Fcal_u(Q)}
+\alpha(Q,k,h)\frac{\sharp\Fcal_u(Q,k,h)}{\sharp\Fcal_u(Q)}
\Biggr)\\
=\lim_{k\to\infty}\lim_{Q\to\infty}
\Biggl(\int f\ud P(Q)+a(Q,k)\Biggr)
\end{multline*}

By Lemma~\ref{ref31},
$$
\int f\ud P(Q) - N(k) \le \int f\ud P(Q) + a(Q,k),
$$
and hence
\begin{multline*}
\biggl(\limsup_{Q\to\infty}\int f\ud P(Q)\biggr)-N(k)=
\limsup_{Q\to\infty}\biggl(\int f\ud P(Q)-N(k)\biggr)\\
\le 
\limsup_{Q\to\infty}\biggl(\int f\ud P(Q)+a(Q,k)\biggr)=
\lim_{Q\to\infty}\biggl(\int f\ud P(Q)+a(Q,k)\biggr).
\end{multline*}
Taking the limit for $k\to\infty$, we get
$$
\limsup_{Q\to\infty}\int f\ud P(Q) \le \int f\ud P.
$$
By a dual argument we obtain
$$
\int f\ud P \le \liminf_{Q\to\infty} \int \ud P(Q),
$$
thus concluding the proof of Theorem~\ref{ref19}.


\end{document}